\documentclass{article}

\usepackage{maa-monthly}
\usepackage{mathtools,amssymb,venndiagram}
\usepackage{tikz}
\usetikzlibrary{shapes,backgrounds}
\usepackage{verbatim}

\theoremstyle{theorem}
\newtheorem{theorem}{Theorem}
\newtheorem{corollary}[theorem]{Corollary}
 
\theoremstyle{definition}
\newtheorem{definition}[theorem]{Definition}
\newtheorem*{classicdef}{The Classic UFD Definition}
\newtheorem*{altdef}{Alternate UFD Definition}
\newtheorem*{examples}{Examples}

\newtheorem*{problem}{Open Problem}

\begin{document}

\title{On the Axioms for a \\ Unique Factorization Domain}
\markright{Axioms for a UFD}
\author{Scott T. Chapman and Jim Coykendall}

\maketitle

{\centering\textit{Dedicated to Professor William W. Smith on his 85th birthday.}\par}

\begin{abstract}
With the growing evolution of the theory of non-unique factorization in integral domains and monoids, the study of several variations to the classical unique factorization domain (or UFD) property have become popular in the literature.   Using one of these variations, the \textit{length-factorial property}, it can be shown that part of the standard classical axioms used in the definition of a UFD is essentially superfluous.
\end{abstract}
\vspace{.2in}

In an introductory abstract algebra class, the following definition plays a key role in the development of the hierarchy of integral domains (see for instance \cite{Fral}, \cite{Gal}, or \cite{Hung}).
\begin{classicdef}
Let $D$ be an integral domain.  $D$ is a \textit{unique factorization domain} (or \textit{UFD}) if each nonzero nonunit of $D$ can be written as a 
product of irreducible elements and whenever
\[
x_1\cdots x_n = y_1\cdots y_m
\]
where each $x_i$ and $y_j$ is irreducible, then \begin{enumerate}
\item $n=m$, and
\item there is a permutation $\sigma$ so that for each $1\leq i\leq n$, we have $x_i = u_i y_{\sigma(i)}$ where $u_i$ is a unit in $D$.
\end{enumerate}
\end{classicdef}
\noindent A deeper dive into the topic generates many other conditions equivalent to the UFD property (also known as the \textit{factorial} property).   Most of these equivalencies can be found in classical
sources (such as the Monthly papers \cite{PMC} and \cite{PS}) and deal with the ideal theory of $D$.  Another good general reference for UFDs is the monograph \cite{Wein}.  In this note, we show, using the relatively new property known as \textit{length-factorial} that the classic definition above can be surprisingly altered so that only irreducible factorizations of equal length have to be checked.  

Over the past 40 years, the study of factorization properties of integral domains, and more generally commutative cancellative monoids, has grown into its own discipline (for instance, see the monographs \cite{newGb} or \cite{GHKb}).  As factorization problems in rings and integral domains are actually factorization problems in the structure's associated multiplicative monoid, the literature in this area has tended to focus on the more general setting of commutative cancellative monoids.  Hence, we can define a commutative cancellative monoid  to be a  \textit{unique factorization monoid (UFM)} (or \textit{factorial monoid}) by making the necessary minor modification in The Classic UFD Definition above.  While we will show that the classical axioms for a UFD can be simplified, this simplification will not hold in the more general monoidal setting.  While the importance of ring (or domain) properties may not be reflected in the recent literature focusing on monoid factorization, these modified axioms demonstrate that ring properties do play a key role in element factorizations and highlight that classical factorization in integral domains can be a surprisingly different arena.   

If $M$ is a commutative cancellative monoid, then let $\mathcal{A}(M)$ represent its set of irreducible elements.  If every nonunit of $M$ can be written as a product of irreducible elements, then $M$ is called \textit{atomic}.  Hence, in the Classic UFD Definition above, a UFD (resp., UFM) is an atomic integral domain (resp., monoid) that satifies properties 1 and 2.  If for a commutative cancellative monoid $M$ we set $M^\times$ to be the set of units of $M$, then the quotient monoid, $M_{\text{red}}= M/M^\times$ has a unique unit.  We call $M_{\text{red}}$ the \textit{reduction} of $M$ and note that by various arguments in \cite{GHKb}, $M_{\text{red}}$ exhibits the same irreducible factorization properties as does $M$.  Hence, without loss throughout  the remainder of our work, we assume that our monoids are reduced.  Under this reduced assumption, we call an irreducible factorization $\alpha_1\cdots \alpha_n=\beta_1\cdots \beta_m$ in $M$ \textit{nondegenerate} if for each $1\leq i\leq n$, $\alpha_i\neq \beta_j$ for each $1\leq j\leq m$.
Thus, in a given nondegenerate irreducible factorization, cancellation of irreducible elements is not possible.

We proceed by focusing on two variations of the factorial property in the more general setting.  The first of our two variations was introduced 50 years ago
by Zaks \cite{Zaks}.  These domains and monoids have been a popular topic in the literature and a survey of related papers can be found in \cite{CC00}. 
\begin{definition}\label{HF}
Let $M$ be an atomic commutative cancellative monoid.  $M$ is a \textit{half-factorial monoid} if whenever 
\[
x_1\ldots x_n = y_1\ldots y_m
\]
where each $x_i$ and $y_j$ is an atom, then $n=m$.
\end{definition}

While the last definition essentially eliminates condition 2 from the UFM definition, our next new class of factorial-like monoids focuses on eliminating (in some sense) condition 1.

\begin{definition}\label{LF}  Let $M$ be an atomic commutative cancellative monoid.  $M$ is \textit{length-factorial} if whenever
\[
x_1\cdots x_n = y_1\cdots y_n
\]
where each $x_i$ and $y_j$ is an atom, then there is a permutation $\sigma$ so that $x_i = y_{\sigma(i)}$ for each $1\leq i\leq n$.
\end{definition}

\noindent As one might expect, we will call an integral domain $D$ half-factorial (resp., length-factorial) if its multiplicative monoid $D-\{0\}$ is half-factorial (resp., length-factorial). We next offer some basic examples.

\begin{examples}\label{EX}  By a well-known theorem of Carlitz \cite{Car}, any ring of algebraic integers of class number 2 is a half-factorial domain that is not a factorial domain.  
If $A$ is a proper subfield of the field $K$, then $A+XK[X] = \{f(X)\,\mid\, f(0)\in A\}$ is a half-factorial domain that is not factorial (If $\alpha\in K$ and $\alpha\not\in A$, then $X^2=X\cdot X = (\alpha X)\cdot (\alpha^{-1}X)$, see \cite[Example 3]{CC00}).  Large classes of Krull monoids are known to be half-factorial; for instance (see \cite{CKO})
\[
U = \{(x_1, x_2, x_3, x_4)\,\mid\, x_1 + x_2 = x_3 + x_4 \mbox{  with each  }x_i\in\mathbb{N}_0\}
\]
where $\mathbb{N}_0$ represents the set of nonnegative integers.
Additional constructions in the language of Dedekind domains can be found in \cite[Propositions 13, 15, and 16]{CC00}.  

With regard to the length-factorial property, any communtative cancellative monoid generated by two elements is length-factorial \cite[Corollary 3.3]{CCGS}.  This condition characterizes length-factorial monoids for the class of Puiseux monoids (additive submonoids of $\mathbb{Q}_0$ \cite[Proposition 4.25]{CGG}) which includes the class of numerical monoids (additive submonoids of $\mathbb{N}_0$).  Length-factorial Krull monoids have been recently characterized with respect to their class groups in \cite[Proposition 1.1]{GZ}.  One simple such example would be the block monoid on the cyclic group $\mathbb{Z}_3$ 
($\mathcal{B}(\mathbb{Z}_3)$), which can be expressed as a Diophantine monoid as 
\[
\{(x_1,x_2, x_3)\,\mid\, x_1+2x_2-3x_3=0\mbox{  with  } x_1, x_2, x_3\in \mathbb{N}_0\}.
\]

We close here by noting that without much additional work, the arguments in \cite{Car} can be used to show
that any ring of algebraic integers of class number greater than 2 is an atomic domain which is neither half-factorial nor length-factorial.

\end{examples} 

The next result follows directly from the three given definitions.  In Figure \ref{F1}, we offer a Venn diagram relationship of our three properties in the general monoidal setting.

\begin{theorem}\label{basicob}  Let $M$ be a commutative cancellative monoid.  Then $M$ is factorial if and only if $M$ is half-factorial and $M$ is length-factorial.
\end{theorem}

\def\firstcircle{[fill=yellow, draw = black](0,0) circle (.8in)}
\def\secondcircle{[fill=cyan, draw = black] (0:1in) circle (.8in)}

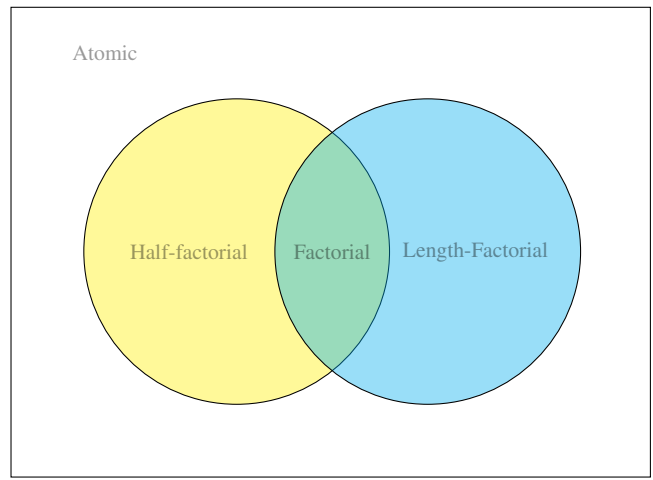
\begin{figure}[h]
\caption{The relationship concerning Factorial properties in general atomic monoids.}
\label{F1}

\begin{center}
\begin{tikzpicture}
\begin{scope} [fill opacity = .4]
\draw [draw=black] (-3,-3) rectangle (5.5,3.25); \node[] at (-1.75,2.65) {\footnotesize{Atomic}};
    \draw \firstcircle;  \node[] at (-.25in,0in) {\footnotesize{Half-factorial}};
    \draw \secondcircle;  \node[] at (1.25in,0in) {\footnotesize{Length-Factorial}}; \node [] at (.5in,0in) {\footnotesize{Factorial}};
    
    \end{scope} 
\end{tikzpicture}
\end{center}
\end{figure}
\vspace{.2in}

It follows easily using the examples cited above that relationships involving Definitions \ref{HF} and \ref{LF} (as illustrated in Figure \ref{F1}) contain no further refinements.  Our goal is to show that in the case of an integral domain, there is a significant refinement.  We will need some definitions
to describe our approach.  If we let 
$\mathcal{F}(\mathcal{A}(M))$ represent the free commutative monoid over $\mathcal{A}(M)$, then we view the factorizations of a nonunit $x\in M$ as elements in the free monoid.  In other words
\[
x=k_1x_1 + \cdots + k_tx_t \mbox{  in  } M
\]
can be viewed in $\mathcal{F}(\mathcal{A}(M))$ as 
\[
\prod_{\mathcal{A}(M)} x_i^{k_i}.
\]

If $M^\bullet$ represents the set of nonunits of $M$, then for $x\in M^\bullet$, set 
\[
Z(x)=\{ \,\prod_{\mathcal{A}(M)} x_i^{k_i} \,\mid\, \sum_{\mathcal{A}(M)} k_ix_i = x\}
\]
to be the set of irreducible (atomic) factorizations of $x$ in $M$ and 
\[ Z(M) =\{\, Z(x)\,\mid\, x\in M^\bullet\,\}\]
to be the complete set of factorizations in $M$.  Notice that there is a unique homomorphism
\[
\pi: Z(M) \rightarrow M
\] where $\pi(x)=x$ for all irreducibles $x\in \mathcal{A}(M)$ called the \textit{factorizaton homomorphism} of $M$.
We further, set
\[
\mathcal{L}(x) = \{ n\,\mid\, n = \sum_{\mathcal{A}(M)}k_i\mbox{  for some  }\prod_{\mathcal{A}(M)} x_i^{k_i} \in Z(x) \}
\]
to be the \textit{set of lengths} of $x$ in $M$.  For simplicity, if $z\in Z(x)$, we set $|z| =  \sum_{\mathcal{A}(M)}k_i$ to be the \textit{length} of the factorization $z$ of $x$.

The \emph{Betti graph} $\nabla_b$ of $b$ is the graph whose set of vertices is $Z(b)$ having an edge between factorizations $z,z' \in Z(x)$ precisely when $z$ and $z'$ share an atom.

\begin{definition}
 An element $b$ of $M$ is a \textit{Betti element} if and only if its Betti graph is disconnected.
\end{definition}

 We let $\text{Betti}(M)$ denote the set of Betti elements of $M$.  Betti elements have proven to be vital in the computation of combinatorial constants related to nonunique factorizations, especially in affine monoids (additive submonoids of $\mathbb{N}_0^k$; see \cite{delta-bf}, \cite{c-t}, \cite{CGSOP}, and \cite{single}) and Puiseux monoids (additive submonoids of $\mathbb{Q}_0$; see \cite{CJMM}, \cite{CG22}, and \cite{GG18}).  A more general treatment of Betti elements can be found in \cite{CGSOP}.  We offer some examples to illustrate the idea behind Betti elements.

\begin{examples}
We return to the monoid $U$ introduced in the previous examples.  Using the existing program structure in the GAP programming language (see \cite{numericalsgps} and \cite{gap}), we obtain that 
$\mathcal{A}(U) = \{ (1,0,0,1), (1, 0, 1, 0),  (0,1,0,1), (0,1, 1, 0)\}$ and $\text{Betti}(U) = \{(1,1,1,1)\}$.   Note that the atomic factorizations of $(1,1,1,1)$
are $(1,0,0,1)+(0,1,1,0) = (1,0,1,0) + (0,1,0,1)$. 

We now use an example that is not half-factorial.  Let $S$ be the numerical monoid generated by 11, 12, 13, 16, 17, 18, and 21 (i.e., $S=\langle 11, 12, 13, 16, 17, 18, 21\rangle$).  Again, using the GAP program, we identify the Betti elements of $S$ as
$\text{Betti}(S)= \{24, 28, 29, 30, 32, 33, 34, 35, 36, 37, 38, 39, 42, 43, 44 \}$.  We use the notation \newline $(x_1, x_2, x_3, x_4, x_5, x_6, x_7)$ to represent the factorization $x_1\cdot 11 + x_2\cdot 12 + x_3\cdot 13 + x_4 \cdot 16 + x_5 \cdot 17 + x_6 \cdot 18 + x_7\cdot 21$.  The factorizations of 43 and 45 in $S$ are as follows:
\[
\begin{array}{c|c}
43 & 45 \\[\smallskipamount]
\hline & \\
z_1=(1,0,0,2,0,0,0) & w_1= (3,1,0,0,0,0,0) \\
z_2=(2,0,0,0,0,0,1) & w_2 = (0,0,1,2,0,0,0) \\
z_3 = (0,0,2,0,1,0,0) & w_3 = (0,1,0,1,1,0,0) \\
z_4 = (0,1,1,0,0,1,0) & w_4 = (1,0,0,0, 2,0,0) \\
 & w_5 = (1,0,0,1,0,1,0) \\
 & w_6 = (0,2,0,0,0,0,1) \\
 & w_7 = (1,0,1,0,0,0,1)
 \end{array}
 \]
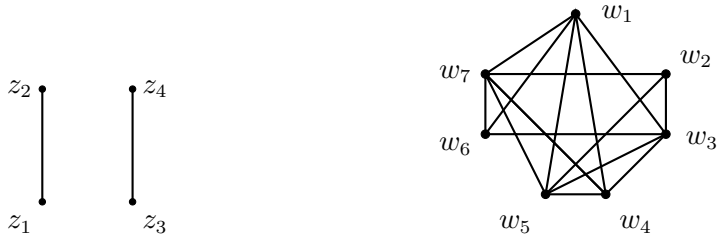
\begin{figure}[ht]
\caption{$\nabla_{43}$ (left) and $\nabla_{45}$ (right) in $S=\langle 11, 12, 13, 16, 17, 18, 21\rangle$}
\label{BettiEle}
\begin{center}
\begin{tikzpicture}[scale=.6]
\draw[fill=black] (1,2.5) circle (2pt);
\draw[fill=black] (3,2.5) circle (2pt);
\draw[fill=black] (1,5) circle (2pt);
\draw[fill=black] (3,5) circle (2pt);
\draw[thick] (1,2.5) -- (1,5);
\draw[thick] (3,2.5) -- (3,5);
\node at (.5,2) {$z_1$};
\node at (.5, 5) {$z_2$};
\node at (3.5, 2) {$z_3$};
\node at (3.5, 5) {$z_4$};
\end{tikzpicture}
\hspace{1.25in}
\begin{tikzpicture}[scale=.8]
\draw[fill=black] (1,0) circle (2pt);
\draw[fill=black] (2,0) circle (2pt);
\draw[fill=black] (0,1) circle (2pt);
\draw[fill=black] (3,1) circle (2pt);
\draw[fill=black] (0,2) circle (2pt);
\draw[fill=black] (3,2) circle (2pt);
\draw[fill=black] (1.5,3) circle (2pt);
\node at (.5,-.5) {$w_5$};
\node at (2.5, -0.5) {$w_4$};
\node at (3.6, .9) {$w_3$};
\node at (3.5, 2.25) {$w_2$};
\node at (2.2, 3) {$w_1$};
\node at (-.5, 2) {$w_7$};
\node at (-.5, .8) {$w_6$};
\draw[thick] (1,0) -- (2,0) -- (3,1) -- (3,2) -- (0,2) -- (1.5,3) -- (3,1) -- (0,1) -- (0,2) -- (1,0) -- (1.5,3) -- (2,0) -- (0,2);
\draw[thick] (0,2) -- (2,0);
\draw[thick] (0,1) -- (1.5,3);
\draw[thick] (1,0) -- (3,2);
\draw[thick] (1,0) -- (3,1);
\end{tikzpicture}

\end{center}
\end{figure}

In Figure \ref{BettiEle}, we exhibit the factorization graphs of the Betti element 43 (on the left) and the non-Betti element 45 (on the right).  
\end{examples}

The length-factorial condition is an extremely strong property and forces a strict restriction on the possible form of a nondegenerate irreducible factorization. We compile a list of conditions equivalent to length-factoriality which have appeared recently in the literature.  In the following, let $\mathcal{Q}(M)$ represent the quotient group (or Grothendieck group, see \cite[Section 2.2]{CCGS}) of $M$.  Moreover, if $\sigma$ is any equivalence relation on $M$, then $\sigma$ is called a \textit{congruence} if respects the operation of $M$ (see again \cite[Section 2.2]{CCGS}). 

\begin{theorem}\label{bigequiv}
Let $M$ be a commutative cancellative atomic monoid.  The following conditions are equivalent.
\begin{enumerate}
\item $M$ is length-factorial and not factorial.
\item  $\text{Betti}(M)=\{x\}$ and $x$ has exactly two factorizations of differing lengths into irreducible elements of $M$.
\item\label{defmfer} There exists in $M$ a nondegenerate irreducible factorization 
\begin{equation}\label{mfer} x=\pi_1^{k_1}\cdots \pi_s^{k_s} = \tau_1^{t_1}\cdots \tau_r^{t_r}\end{equation} 
where $\sum_{1\leq i\leq s} k_i > \sum_{1\leq j\leq r} t_j \vspace{3pt}$ and every nondegenerate irreducible factorization of 
$M$ is of the form $x^n$ for $n\geq 1$.
\item There exists an $x\in\mathcal{A}(M)$ such that $\mathcal{A}(M)\backslash \{a\}$
and $a- \mathcal{A}(M)\backslash \{a\}$ are integrally independent sets in $\mathcal{Q}(M)$.
\item The congruence $\text{ker}\, \pi=\{\, (z_1,z_2)\in Z(M)^2\,\mid\, \pi(z_1)=\pi(z_2)\,\}$ on $Z(M)$ is cyclic and nontrivial.
\end{enumerate}
\end{theorem}

The factorization \eqref{mfer} is referred to in \cite{CS11} as the \textit{master factorization} of $M$.  We also note the importance in condition 2 that the Betti element has exactly two irreducible factorizations of differing length.  The monoid $U$ introduced in our first set of examples has a unique Betti element, but is not length-factorial.  

\begin{proof} We note that the equivalence of 1 and 3 is due to \cite[Proposition 2.9]{CS11}, the equivalence of 1 and 2 is due to \cite[Theorem 4.7]{CGSOP}, and the equivalence of 1, 4, and 5 is established in \cite[Theorem 3.1]{CCGS}.  The reader is directed to those sources for precise proofs.  
\end{proof}



Based on the master factorization property, we can make an important deduction concerning length-factorial monoids that are not factorial.  Some past definitions from the literature will be of interest here.  An integral domain $D$ which is atomic and has finitely many atoms is called a \textit{Cohen-Kaplansky domain} (first studied in \cite{CK46} and then more extensively in \cite{CK} in which the terminology was defined).  In a broader sense, an atomic domain $D$ which has finitely many nonassociated nonprime irreducible elements is known as a \textit{generalized Cohen-Kaplansky domain} (or \textit{generalized CK-domain}) and are explored in detail in \cite{AAZ}.  We extend these definitions to atomic monoids in the obvious manner. 

\begin{corollary}\label{conseqlf}  Let $M$ be an atomic monoid which is length-factorial and not factorial with master factorization \eqref{mfer}.  Every nonprime irreducible of $M$ appears in \eqref{mfer}, and hence $M$ must have finitely many nonprime irreducible elements.  Consequently, $M$ is a generalized CK-monoid.
\end{corollary}

\begin{proof}
By definition, every nonprime irreducible of $M$ appears in a nondegenerate irreducible factorization, so it must appear somewhere in the master factorization.  Hence, there can be at most finitely many nonassociated nonprime irreducibles. 
\end{proof}




The last corollary allows us to improve Theorem \ref{basicob} for integral domains.

\begin{theorem}\label{thebigone}
An integral domain $D$ is factorial if and only if it is length-factorial.
\end{theorem}

The original proof of this result from \cite{CS11} considers cases predicated on the form of the master factorization on a case-by-case basis.  We offer a different proof using first principles from commutative algebra which reduces $D$ to the much simpler local case and exploits some basic properties of CK-domains.

\begin{proof}  If $D$ is factorial, then clearly $D$ is length-factorial.  Hence, we assume that $D$ is length-factorial and assume for the sake of contradiction that it is not factorial. 

Since $D$ is length-factorial, Corollary \ref{conseqlf} shows that there are only finitely many nonprime irreducibles (and at least one, as $D$ is not factorial), and so if $S$ is the multiplicative set of all elements that can be factored into primes (or units) of $D$, then 
\[R:=D_S=\{\,\frac{a}{b}\,\mid\, a\in D\mbox{ and }b\in S\,\}\] 
is still length-factorial and is also a CK-domain. To further this reduction, we note that \cite[Theorem 2 and Theorem 7]{CK46} show that 
\begin{itemize}
    \item $R$ has finitely many maximal ideals, 
    \item any irreducible of $R$ lies in only one prime (maximal) ideal of $R$, and 
    \item there is a one-to-one correspondence between irreducibles of $R$ contained in the maximal ideal $M$ and the irreducibles of the localization $R_M$ (up to associates in each ring).
    \end{itemize}
    These reductions allow us to assume that we have a non-factorial, length-factorial domain $R$ that is local (i.e., has a unique maximal ideal). Additionally, \cite[Theorem 4.3]{CK} shows that if $R$ is a local CK-domain, then the integral closure, $\overline{R}$, of $R$ is a local PID and there is an $n\in\mathbb{N}$ such that $x^n\in R$ for all $x\in\overline{R}$.

 Using the notation above, we let $R$ be our local length-factorial CK-domain, $\overline{R}$ its integral closure, a PID with maximal ideal generated by $\pi$, and $n\in\mathbb{N}$ such that $x^n\in R$ for all $x\in\overline{R}$. If $\alpha$ and $\beta$ are nonzero nonunits in $\overline{R}$, we can write 
 \[\alpha=u\pi^a\mbox{ and }\beta=v\pi^b\] 
 where $u,v$ are units in $\overline{R}$ and $a,b\in\mathbb{N}$. So we select $\alpha$ and $\beta$ non-associate irreducibles such that $a$ and $b$ are minimized. We first observe that if $a=b$, then $\alpha^n$ is associated (in $R$) with $\beta^n$ which contradicts length-factoriality. If, on the other hand (and without loss of generality), we have $a<b$, then we call on the additive structure of $R$ (which bifurcates from the smoother ``monoid only" case). Note that $\alpha+\beta$ can be written (as an element of $\overline{R}$) as $\pi^a(u+v\pi^{b-a})$. Since $\overline{R}$ is also local, $u+v\pi^{b-a}$ is a unit in $\overline{R}$ and by the minimality of $a$, $\alpha+\beta$ is also irreducible.  Thus, $\alpha$ and $\alpha + \beta$ are irreducibles of $R$, and since $a<b$ are minimal, this means $\alpha$ and $\alpha+\beta$ must be associated in $R$ and hence $\alpha$ divides $\beta$ which is our desired contradiction.
\end{proof}

Here is an alternate proof of Theorem \ref{thebigone} which again uses the properties of a generalized CK-domain as well as some past results from the theory of nonunique factorizations.  A atomic monoid $M$ is said to be \textit{square-factorial} if for atoms $u, v$, and $w$ of $M$, $u^2=vw$ implies that $u$ and $v$ are associates (see \cite[p. 207]{ACHKZ}).  Notice that if $M$ is not square-factorial, then it is not length-factorial.  If $D$ is an atomic length-factorial domain, then Corollary \ref{conseqlf} establishes that it is a generalized CK-domain.  A generalized CK-domain that is not factorial is not square-factorial by \cite[Theorem 4.6]{ACHKZ}.  This completes the argument.  


\def\secondcirclea{[fill=cyan, draw = black] (.85,0) circle (.5in)}
\def\firstcirclea{[fill=yellow, draw = black](.85,0) circle (1in)}

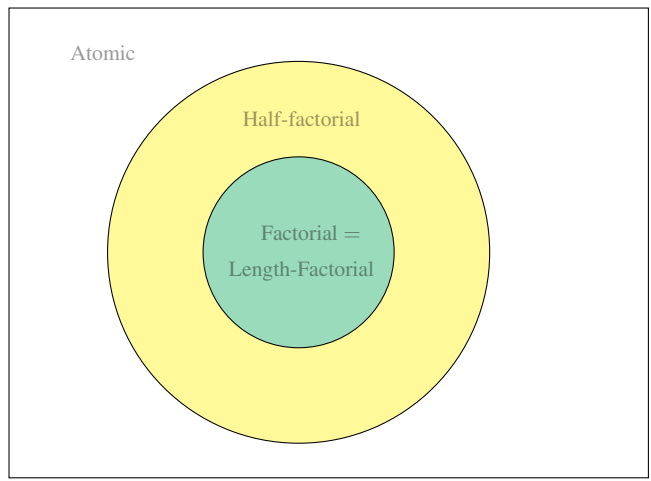
\begin{figure}[h]
\caption{The relationship concerning Factorial properties in atomic integral domains.}

\begin{center}
\begin{tikzpicture}
\begin{scope} [fill opacity = .4]
\draw [draw=black] (-3,-3) rectangle (5.5,3.25); \node[] at (-1.75,2.65) {\footnotesize{Atomic}};
    \draw \firstcirclea;  \node[] at (.35in,.7in) {\footnotesize{Half-factorial}};
    \draw \secondcirclea;  \node[] at (.35in,-.1in) {\footnotesize{Length-Factorial}}; \node [] at (.4in,.1in) {\footnotesize{Factorial  }$=$};
    
  \end{scope} 
\end{tikzpicture}
\end{center}
\end{figure}

As a corollary to Theorem \ref{thebigone}, we obtain that a non-factorial integral domain always contains a nondegenerate irreducible factorization of equal length.

\begin{corollary}[Bonus]  Let $D$ be an atomic integral domain which is not factorial.  There exists a positive integer $n\geq 2$ and irreducible elements $\alpha_1, \ldots, \alpha_n, \beta_1, \ldots, \beta_n$ such that 
\[
\alpha_1 \cdots \alpha_n= \beta_1 \cdots \beta_n
\]
is a nondegenerate irreducible factorization in $D$.
\end{corollary}

In fact, we have now shown that the following is an equivalent defintion of the UFD property.

\begin{altdef}
Let $D$ be an integral domain in which every nonzero nonunit of $D$ can be factored as a product of irreducibles.  $D$ is a \textit{UFD} if whenever
\[
x_1\cdots x_n = y_1\cdots y_n
\]
where each $x_i$ and $y_j$ is irreducible, then there is a permutation $\sigma$ so that for each $1\leq i\leq n$, we have $x_i = u_iy_{\sigma(i)}$ where $u_i$ is a unit in $D$.
\end{altdef} 

We note that the initial condition in both the Classical and Alternate Definitions, that the domain be atomic, is essential.  While not emphasized in an introductory course, there are relatively simple examples of integral domains that fail this property.  One such example, 
\[
\mathbb{Z} + X\mathbb{Q}[X] = \{ f(X)\,\mid\, f(X)\in \mathbb{Q}[X]\mbox{  and  }f(0)\in \mathbb{Z}\},
\]
can be found (with a detailed explanation) in \cite[Example p. 166]{ChapMM}.

Can the Alternate UFD Definition be further improved?  In \cite{ACHKZ} the authors explore several types of monoids and integral domains which generalize the factorial property.  Here is one that is particularly relevant to our current discussion.  

\begin{definition}  Let $M$ be an atomic commutative cancellative monoid and suppose that $n\geq 2$ is a positive integer.  We denote $M$ to be \textit{quasi-}$n$\textit{-factorial} if whenever
\[
\alpha_1\cdots \alpha_n=\beta_1\cdots \beta_n\]
for atoms $\alpha_1,\ldots, \alpha_n,\beta_1, \ldots, \beta_n$, then there is a permutation $\sigma$ so that $\alpha_i$ and $\beta_{\sigma(i)}$ are associates.   
\end{definition}

Notice by the definition (and its obvious extension) that an atomic monoid $M$ is length-factorial if and only if it is quasi-$n$-factorial for all positive integers $n>1$.  There are several examples constructed in \cite{ACHKZ} of atomic monoids that are quasi-$n$-factorial and neither factorial nor length-factorial (see Examples 2.3, 2.4, and 2.5).  The paper goes on to explore conditions on Krull (and more specifically Dedekind) domains which would force such a domain to be quasi-$n$-factorial and not factorial.  No specific construction is offered and the existence of such a domain remains open.  Hence, we close with the following problem.

\begin{problem}  Does there exist an atomic integral domain that is quasi-$n$-factorial for some $n\geq 2$ but not factorial?
\end{problem}

\noindent Notice that if the answer is ``no,'' then the $n$ in the Alternate UFD Definition could be changed to any specific integer you like (including 2)! 

\medskip\medskip 

\noindent\textsc{Funding:} No Funding was granted to the information in this paper.
\medskip

\noindent\textsc{Author Contribution Statement:} All the authors, Scott T. Chapmand and Jim Coykendall, contributed to this manuscript, wrote the manuscript and were involved with project concept and submission.  They ensure that they meet the criteria for authorship as per the ICMJE guidelines.
\medskip

\noindent\textsc{Conflict of Interst Statement:} Scott Chapman reports that he has no conflicts of interest.  James Coykendall reports that he has no conflicts of interest.

\begin{acknowledgment}{Acknowledgment.}
The authors thank Pedro Garc\'ia-S\'anchez and Alfred Geroldinger for discussions related to this work.
\end{acknowledgment}

\begin{biog}
\item[Scott T. Chapman] is a Texas State University System Regents’ Professor and SHSU Distinguished Professor
at Sam Houston State University in Huntsville, Texas. He served as Editor of the \textit{American Mathematical Monthly} during the period 2012-2016.  He is currently serving as Editor-in-Chief at \textit{Communications in Algebra}.  His editorial
work, numerous publications in the area of non-unique factorizations, and years of
directing REU Programs, led to his designation in 2017 as a Fellow of the American
Mathematical Society.
\begin{affil}
Department of Mathematics and Statistics, Sam Houston State University, Box 2206, Huntsville, TX  77341\\
scott.chapman@shsu.edu
\end{affil}

\item[James B. (Jim) Coykendall] received his BS from Caltech and a Ph. D. degree in Algebraic Number Theory from Cornell.   He has held Full Professor appointments at North Dakota State University and Clemson University and served as Department Chair at both institutions.   He has about 60 publications in various areas of commutative algebra and number theory.  He is the founding Editor and current Editor-in-Chief of \textit{Journal of Commutative Algebra}, and has also served on the Editorial Boards at both \textit{The Rocky Mountain Journal} and \textit{Communications in Algebra}.  He is the recipient of numerous teaching awards, including being named the Carnegie Foundation’s Teacher of the Year in North Dakota in 2005.   For more than 20 years, he has also made a significiant contribution to the management of the Mathematical Genealogy Project and played a vital role in forging a permanent agreement with the AMS to manage the Project in 2010. 
\begin{affil}
Department of Mathematics and Statistical Sciences, Clemson University, Clemson, SC  29634\\
jcoyken@clemson.edu
\end{affil}
\end{biog}
\vfill\eject

\end{document}